\pdfoutput=1
\documentclass[11pt]{article}

\usepackage{amsmath,amssymb,amsthm,mathtools}
\usepackage{geometry}
\usepackage{graphicx}
\usepackage{subcaption}
\usepackage{booktabs}
\usepackage{array}
\usepackage{float}
\usepackage{hyperref}
\theoremstyle{plain}
\numberwithin{equation}{section}

\newtheorem{example}{Example}[section]
\newtheorem{theorem}{Theorem}[section]
\newtheorem{lemma}[theorem]{Lemma}
\newtheorem{proposition}[theorem]{Proposition}
\newtheorem{corollary}[theorem]{Corollary}

\theoremstyle{definition}
\newtheorem{definition}[theorem]{Definition}

\theoremstyle{remark}
\newtheorem{remark}[theorem]{Remark}

\providecommand{\subjclass}[2][]{\par\noindent\textbf{2020 Mathematics Subject Classification.} #2}
\providecommand{\keywords}[1]{\par\noindent\textbf{Keywords.} #1}

\title{Preimage Regions of Symmetric Separable Maps on the Simplex:
	Convexity and Barycentric Star-Shapedness
	\thanks{Corresponding author. Email: xcui@nju.edu.cn\\
		Xiaojun Cui and Jilong Xu are supported by the National Natural Science Foundation
		of China (Grant No. 12171234), the Project Funded by the Priority Academic Program
		Development of Jiangsu Higher Education Institutions (PAPD) and the Fundamental
		Research Funds for the Central Universities.}}
\author{Jilong Xu$^a$, Xiaojun Cui$^{a,*}$\\
	\small\em $^a$ School of Mathematics, Nanjing University, Nanjing {\rm 210093}, P. R. China}
\date{}

\begin{document}
\maketitle

\begin{abstract}
We study preimage regions on the open probability simplex associated with
symmetric separable functionally generated maps.  The problem is a
finite-dimensional geometric question about convexity and barycentric
star-shapedness of these regions.  In the portfolio interpretation, the
regions consist of the points whose generated portfolio has no negative
coordinate.

For symmetric separable generators, the defining first-order inequalities
split into a coordinate term and a symmetric aggregation term.  This
coordinate--aggregation decomposition is the main organizing device of the
paper.  We show that the aggregation term may destroy convexity, and may even
destroy barycentric star-shapedness.  In particular, moving closer to the
barycenter need not preserve the long-only property.

We then give a necessary and sufficient threshold criterion for barycentric
star-shapedness and derive sufficient conditions that recover it.  These
conditions are expressed in terms of concavity and second-derivative
domination for the aggregation function.  The entropy case is the affine
aggregation case, in which the long-only constraints reduce to coordinate
thresholds.
\end{abstract}

\subjclass[2020]{Primary 52A30; Secondary 26B25, 91G10}
\keywords{simplex; symmetric separable generators; preimage regions; convexity; barycentric star-shapedness; Jensen differences; functionally generated portfolios}

\section{Introduction}
\label{sec:introduction}
This paper studies a concrete class of regions on the probability simplex.
They are the preimages of the nonnegative-coordinate constraint under
symmetric separable maps arising from functionally generated portfolios.  In
portfolio terminology, they consist of the market weights for which the
generated portfolio has no negative coordinate.  The problem studied here is
their convexity and their star-shapedness with respect to the barycenter.

The financial motivation comes from Stochastic Portfolio Theory.  A
functionally generated portfolio assigns portfolio weights to market weights
through a gradient-type formula associated with a generating function, starting
with Fernholz's portfolio generating functions
\cite{Fernholz1999PortfolioGeneratingFunctions,Fernholz2002SPT}.  In discrete
time, Wong, partly in joint work with Pal, developed a finite-dimensional
framework in which generated portfolios are described by concave functions and
first-order inequalities on the simplex
\cite{Wong2015OptimizationofRelativeArbitrage,
	PalWong2016GeometryRelativeArbitrage,Wong2019InfromationGeometryinPortfolio}.
In the standard long-only setting, a positive concave generating function on
the whole simplex produces a portfolio with nonnegative coordinates; see
\cite[Proposition 5]{PalWong2016GeometryRelativeArbitrage}.

In this paper the generating function is considered on its positive domain
\(D_\Phi=\{\mu\in\Delta^{(n)}:\Phi(\mu)>0\}\), which may be a proper subset
of the simplex.  On this domain the generated portfolio is well defined, but
some of its coordinates may be negative.  This restricted-domain viewpoint is
related to the admissible long-short framework of
\cite{Xu2026GeometryofAdmissibleShortSelling}, although no result from that
work is used here.  In the present paper we restrict attention to the symmetric
separable class and study whether the resulting long-only preimage is convex
and whether the whole segment from the barycenter to any long-only point
remains long-only.

This center-directed property is star-shapedness with respect to the
barycenter \cite{HansenHerburtMartiniMoszynska2020StarshapedSets}.  It is
natural here because the generators are symmetric and the barycenter is the
distinguished point of the simplex.  In the portfolio motivation, it
corresponds to the heuristic that moving a market weight toward the
equal-weight point should reduce the risk of producing a negative coordinate.

We show that this heuristic is not generally valid.  Even for symmetric
separable concave generators, the long-only preimage region may be nonconvex and may
fail to be star-shaped with respect to the barycenter.  The reason is visible
from the symmetric separable normal form of the constraints.  Each inequality splits
into a coordinate term and a symmetric aggregation term.  Along a segment, the
coordinate term relevant to a constraint may stay fixed while the aggregation
term changes.  A midpoint Jensen difference of the aggregation function can
then make a midpoint fail the long-only condition even when the endpoints satisfy it.

The same decomposition also explains when the barycentric behavior can be
recovered.  Concavity of the aggregation function makes barycentric movement
improve the aggregation term.  More generally, a second-derivative domination
condition allows some unfavorable behavior of the aggregation term to be
offset by the coordinate term.  The entropy case is the affine aggregation
case: the aggregation term is constant on the simplex and the long-only
constraints reduce to coordinate thresholds.

Although the terminology comes from portfolio generation, the results below
are finite-dimensional and deterministic.  They do not prove relative
arbitrage, outperformance, or an optimization theorem.  The contribution is a
geometric analysis of a family of long-only preimage regions on the simplex.
The symmetric separable setting is used as a benchmark class because it gives
a complete coordinate--aggregation decomposition while still exhibiting
nonconvexity, failure of barycentric star-shapedness, and recovery cases.

\subsection*{Main results}

The main results follow the preceding decomposition.  First, Theorem
\ref{thm:local-jensen-nonconvexity} gives a midpoint Jensen-difference
mechanism for nonconvexity.  For suitable vertical shifts, two long-only points
can have a midpoint that is not long-only, even though that midpoint is closer to the
barycenter.

Second, Theorem \ref{thm:sharp-barycentric-threshold} gives a necessary and sufficient criterion
for barycentric star-shapedness of all shifted regions \(\mathcal L_T\).  The
criterion is expressed through the minimal constraint function
\(\Theta_\psi\), and a symmetric separable example shows that the criterion can
fail.

Third, Theorems \ref{thm:star-shaped-concave-eta} and
\ref{thm:curvature-domination-barycentric} give recovery conditions.  Concavity
of the aggregation function is sufficient, and a quantitative
second-derivative domination condition gives a more flexible criterion.
Corollary \ref{cor:nonconvex-but-star-shaped} shows that convexity and
barycentric star-shapedness are distinct properties.

The entropy case is treated as a boundary example: affine aggregation makes the
constraints collapse to coordinate thresholds.

\subsection*{Organization of the paper}

Section \ref{sec:general-setting} gives the simplex setup and the symmetric
separable normal form.  Section \ref{sec:intuition} proves the midpoint
mechanism and records quadratic and power-family examples.  Section
\ref{sec:star-shaped} gives the barycentric threshold criterion and a failure
example.  Section \ref{sec:recovery-of-barycentric-stability} proves the
recovery results.  Section \ref{sec:entropy} treats the affine entropy case,
and Section \ref{sec:conclusion-open-problems} concludes.  Auxiliary proofs
are placed in the appendices.

\section{General setting}\label{sec:general-setting}
This section fixes the deterministic simplex framework used throughout the
paper.  We first define the long-only preimage region induced by the nonnegative
simplex constraint, and then specialize to the symmetric separable class.  In
that class we introduce the threshold parameter, the aggregation function, and
the minimal constraint function that will organize the rest of the paper.
\subsection{Long-only preimage region}

Let \(n\ge 2\), and consider the open probability simplex
\[
\Delta^{(n)}=
\Bigl\{\mu=(\mu_1,\dots,\mu_n)\in\mathbb{R}^n:
\mu_i>0,\ \sum_{i=1}^n \mu_i=1\Bigr\}.
\]
In the portfolio interpretation, \(\mu\) is a vector of market weights.  For
the geometric analysis below, it is simply a point of the simplex.

Let \(\Phi\) be a differentiable concave function on an open convex subset of
\(\Delta^{(n)}\).  We write
\[
D_\Phi:=\{\mu\in\Delta^{(n)}:\Phi(\mu)>0\}
\]
for its positive region, and assume that \(D_\Phi\) is nonempty.  We use the following gradient convention.  Extend \(\Phi\) smoothly to a
neighborhood in \(\mathbb R^n\), take the usual Euclidean gradient of this
extension, and project it onto the tangent space
\[
\Bigl\{v\in\mathbb R^n:\sum_{i=1}^n v_i=0\Bigr\}.
\]
This projected vector is denoted by \(\nabla\Phi(\mu)\).  Since
\(e_i-\mu\) lies in the same tangent space, this is the gradient relevant for
the first-order formula below.  On \(D_\Phi\), the regularized-gradient
formula defines weights
\begin{equation}\label{eq:gradient-formula}
	\pi_i(\mu)
	=
	\mu_i\left(
	1+\frac{\langle\nabla\Phi(\mu),e_i-\mu\rangle}{\Phi(\mu)}
	\right),
	\qquad i=1,\dots,n,
\end{equation}
where \(e_i\) is the \(i\)-th vertex of the simplex.  Since
\(\sum_{i=1}^n\mu_i(e_i-\mu)=0\), the vector \(\pi(\mu)\) has total mass one.

The map \(\mu\mapsto\pi(\mu)\) is the generated portfolio associated with
\(\Phi\).  In the long-short framework of
\cite{Xu2026GeometryofAdmissibleShortSelling}, whose preprint is publicly
available, the positive region of the generating function is allowed to be a
proper subset of the simplex.  Then
\eqref{eq:gradient-formula} still defines a unit-mass generated portfolio on
\(D_\Phi\), but its coordinates need not be nonnegative.  The long-only region
is precisely the preimage of the nonnegative simplex under this map.  Since we do not restrict attention to generators whose portfolios are
long-only on all of \(D_\Phi\), we refer to such generated portfolios as
long-short in the present paper.

The analysis below is finite-dimensional and deterministic.  It is motivated
by discrete-time stochastic portfolio theory, but it does not require a
continuous-time stochastic-calculus framework.  We study the geometry of the
first-order inequalities induced by \eqref{eq:gradient-formula}.  For
\(\mu\in D_\Phi\), define
\[
G_i(\mu):=\Phi(\mu)+\langle \nabla\Phi(\mu),e_i-\mu\rangle,
\qquad i=1,\dots,n.
\]
Since \(\mu_i>0\) and \(\Phi(\mu)>0\), the long-only condition is equivalent
to
\[
\pi_i(\mu)\ge0
\iff
1+\frac{\langle\nabla\Phi(\mu),e_i-\mu\rangle}{\Phi(\mu)}\ge0
\iff
G_i(\mu)\ge0.
\]
Thus the long-only region with respect to \(\Phi\) is
\begin{equation}\label{eq:long-only-region}
	\mathcal L
	=
	\{\mu\in\Delta^{(n)}:\Phi(\mu)>0,\ G_i(\mu)\ge0,\ i=1,\dots,n\}.
\end{equation}
Here and below, \(\Phi(\mu)>0\) is understood only at points where
\(\Phi\) is defined.  When the word feasible is used, it always means
long-only feasible in this sense: the generated portfolio has no negative
coordinate.

Thus \(\mathcal L\) is the preimage of the nonnegative simplex under the
generated portfolio map; equivalently, it is the set of market weights for
which the generated portfolio has no negative coordinate.  In the symmetric
separable setting considered below, the barycenter is the distinguished center.
We study whether this preimage is convex, meaning that the long-only property
is preserved under interpolation, and whether it is star-shaped with respect
to the barycenter, meaning that movement from a long-only point toward
\(\bar e\) remains long-only.

\subsection{Symmetric separable generators and the normal form}
\label{sec:sym-separable-normal-form}

We now restrict attention to a symmetric separable class.  These generators
are invariant under coordinate permutations, so the barycenter
\(\bar e=(1/n,\dots,1/n)\) is the natural reference point.

\begin{definition}[Symmetric separable generator]
	Let \(B\in\mathbb R\) and let \(\psi\in C^2(0,1)\) be concave.  A function
	of the form
	\[
	\Phi_B(\mu)=B+\sum_{i=1}^n\psi(\mu_i)
	\]
	is called a \emph{symmetric separable generator} whenever its positive
	region
	\begin{equation}\label{eq:positive-region-PhiB}
		D_{\Phi_B}:=\{\mu\in\Delta^{(n)}:\Phi_B(\mu)>0\}
	\end{equation}
	is nonempty.
\end{definition}

The following identity is the basic normal form of the long-only constraints.

\begin{lemma}[Normal form of the long-only constraints]
	\label{lem:normal-form}
	For a symmetric separable generator, define the aggregation function
	\begin{equation}\label{eq:eta-definition}
		\eta_\psi(t):=\psi(t)-t\psi'(t).
	\end{equation}
	Then, for \(\mu\in D_{\Phi_B}\),
	\begin{equation}\label{eq:Gi-normal-form}
		G_i(\mu)=B+\psi'(\mu_i)+\sum_{j=1}^n\eta_\psi(\mu_j),
		\qquad i=1,\dots,n.
	\end{equation}
\end{lemma}

\begin{proof}
	Substituting \(\partial_i\Phi_B(\mu)=\psi'(\mu_i)\) into
	\(G_i(\mu)=\Phi_B(\mu)+\partial_i\Phi_B(\mu)-\sum_j\mu_j\partial_j\Phi_B(\mu)\)
	gives \eqref{eq:Gi-normal-form}.
\end{proof}

\begin{remark}[Coordinate term and aggregation term]
	Formula \eqref{eq:Gi-normal-form} separates the \(i\)-th constraint into a
	coordinate term \(\psi'(\mu_i)\) and a symmetric aggregation term
	\(\sum_{j=1}^n\eta_\psi(\mu_j)\).
	
	The coordinate term detects the size of the \(i\)-th component, whereas the
	aggregation term can change along line segments even when the relevant
	coordinate is fixed.  This is the mechanism behind the Jensen-difference
	nonconvexity results and the barycentric inequalities developed below.
\end{remark}

\subsection{Threshold regions and the minimal constraint}
\label{subsec:threshold-minimal-constraint}

For shifted separable generators it is convenient to use a threshold
parameter rather than the vertical shift itself.  We write the unshifted
separable sum as \(S\) and use \(T\) as the primary parameter; the vertical
shift notation \(B=-T\) is mentioned only when needed to connect with
\(\Phi_B\).  This convention reduces the long-only problem to a family of
threshold regions with a common separable shape.

\begin{definition}[Unshifted constraints and threshold regions]
	\label{def:unshifted-threshold-regions}
	For a concave function \(\psi\in C^2(0,1)\), define
	\[
	S(\mu):=\sum_{j=1}^n\psi(\mu_j),
	\qquad
	\eta_\psi(t):=\psi(t)-t\psi'(t).
	\]
	The corresponding unshifted first-order constraints are
	\[
	S_i(\mu):=
	S(\mu)+\langle\nabla S(\mu),e_i-\mu\rangle
	=
	\psi'(\mu_i)+\sum_{j=1}^n\eta_\psi(\mu_j),
	\qquad i=1,\dots,n.
	\]
	For \(T\in\mathbb R\), define the threshold long-only region
	\begin{equation}\label{eq:threshold-long-only-region}
		\mathcal L_T
		:=
		\{\mu\in\Delta^{(n)}:S(\mu)>T,\ S_i(\mu)\ge T,\ i=1,\dots,n\}.
	\end{equation}
	Equivalently, \(\mathcal L_T\) is the long-only region of the shifted
	generator \(\Phi_{-T}(\mu)=S(\mu)-T\), whose positive region is
	\[
	D_{\Phi_{-T}}=\{\mu\in\Delta^{(n)}:S(\mu)>T\}.
	\]
\end{definition}

Throughout the rest of the paper, for \(\mu\in\Delta^{(n)}\), we write
\(\mu_{(1)}:=\max_{1\le i\le n}\mu_i\), using the standard rank notation in
stochastic portfolio theory; see \cite{Fernholz2002SPT}.

\begin{lemma}[The largest coordinate determines the minimal constraint]
	\label{lem:largest-coordinate-active}
	Let \(\psi\) be concave, and let \(S_i\) be the unshifted constraints from
	Definition \ref{def:unshifted-threshold-regions}.  Then
	\[
	\min_{1\le i\le n}S_i(\mu)
	=
	\sum_{j=1}^n\eta_\psi(\mu_j)+\psi'(\mu_{(1)}).
	\]
\end{lemma}

\begin{proof}
	By the normal form, only the term \(\psi'(\mu_i)\) depends on \(i\).  Since
	\(\psi\) is concave, \(\psi'\) is nonincreasing, so the minimum is attained
	at a largest coordinate.
\end{proof}

\begin{definition}[Minimal constraint function]
	\label{def:minimal-constraint-function}
	For a symmetric separable generator, define
	\begin{equation}\label{eq:theta-definition}
		\Theta_\psi(\mu)
		:=
		\min_{1\le i\le n}S_i(\mu)
		=
		\sum_{j=1}^n\eta_\psi(\mu_j)+\psi'(\mu_{(1)}).
	\end{equation}
\end{definition}

Thus the threshold long-only region from
Definition \ref{def:unshifted-threshold-regions} can be written in terms of
the minimal constraint function as
\begin{equation}\label{eq:LT-active-form}
	\mathcal L_T
	=
	\{\mu\in\Delta^{(n)}:S(\mu)>T,\ \Theta_\psi(\mu)\ge T\}.
\end{equation}
Equivalently, \(\mathcal L_T\) is the long-only region of the shifted
generator \(\Phi_{-T}(\mu)=S(\mu)-T\) on its positive region.

\section{The midpoint mechanism and failure of convexity}
\label{sec:intuition}
This section develops the near-barycenter mechanism by which convexity can fail.  The
key point is that, along a chord, the coordinate term in a long-only
constraint may remain fixed while the aggregation term changes by a midpoint
Jensen difference.  This produces explicit threshold intervals for which the
long-only region is nonconvex.
\subsection{Nonconvexity near the barycenter}

In the symmetric setting, one might expect short selling to occur only far
from the barycenter.  The normal form shows why this intuition can fail:
along a segment, a coordinate term may remain fixed while the aggregation
term \(\sum_j\eta_\psi(\mu_j)\) changes.  The relevant variation is measured
by a midpoint Jensen difference of the aggregation function \(\eta_\psi\).

Fix \(n\ge3\), put \(b=1/n\), and for \(0<\varepsilon<b/2\) set
\(a_\varepsilon=b+\varepsilon\), \(c_\varepsilon=b-2\varepsilon\), and
\(d_\varepsilon=(a_\varepsilon+c_\varepsilon)/2=b-\varepsilon/2\).  We use
the three points
\begin{equation}\label{eq:pe-qe-me}
	\begin{aligned}
		p_\varepsilon&=(a_\varepsilon,a_\varepsilon,c_\varepsilon,b,\ldots,b),\\
		q_\varepsilon&=(c_\varepsilon,a_\varepsilon,a_\varepsilon,b,\ldots,b),\\
		m_\varepsilon&=\frac{p_\varepsilon+q_\varepsilon}{2}
		=(d_\varepsilon,a_\varepsilon,d_\varepsilon,b,\ldots,b),
	\end{aligned}
\end{equation}
where the trailing coordinates are omitted when \(n=3\).

\begin{lemma}[The midpoint is closer to the barycenter]
	\label{lem:midpoint-closer}
	For the points in \eqref{eq:pe-qe-me},
	\[
	\|p_\varepsilon-\bar e\|^2=\|q_\varepsilon-\bar e\|^2=6\varepsilon^2,
	\qquad
	\|m_\varepsilon-\bar e\|^2=\frac32\varepsilon^2.
	\]
	Thus \(m_\varepsilon\) is strictly closer to \(\bar e\) than both endpoints.
\end{lemma}

\begin{proof}
	Only the first three coordinates differ from \(b\).  Their deviations are
	\((\varepsilon,\varepsilon,-2\varepsilon)\),
	\((-2\varepsilon,\varepsilon,\varepsilon)\), and
	\((-\varepsilon/2,\varepsilon,-\varepsilon/2)\), respectively.  Squaring and
	summing gives the claim.
\end{proof}

\begin{definition}[Midpoint Jensen difference]
	For \(f\in C^2(I)\) and \(x-h,x+h\in I\), define
	\[
	J_f(x,h):=f(x+h)+f(x-h)-2f(x).
	\]
	For convex \(f\), this is the usual nonnegative midpoint Jensen gap; here we
	use the signed form.  The terminology is consistent with the literature on
	Jensen differences and Jensen-type functionals, especially in connection
	with convexity and entropy inequalities
	\cite{BurbeaRao1982HigherOrderJensen,SahooWong1988GeneralizedJensenDifference}.
\end{definition}

\begin{lemma}[Jensen-difference identity in arbitrary dimension]
	\label{lem:jensen-difference-general}
	For any symmetric separable generator,
	\begin{equation}\label{eq:jensen-difference-general}
		G_2(p_\varepsilon)-G_2(m_\varepsilon)
		=
		\eta_\psi(a_\varepsilon)+\eta_\psi(c_\varepsilon)
		-2\eta_\psi(d_\varepsilon)
		=
		J_{\eta_\psi}\left(d_\varepsilon,\frac{3\varepsilon}{2}\right),
	\end{equation}
	and the same identity holds with \(p_\varepsilon\) replaced by
	\(q_\varepsilon\).
\end{lemma}

\begin{proof}
	In the normal form \eqref{eq:Gi-normal-form}, the terms \(B\),
	\(\psi'(\mu_2)\), the contribution \(\eta_\psi(\mu_2)\), and all trailing
	\(\eta_\psi(b)\)-terms are the same for \(p_\varepsilon\) and
	\(m_\varepsilon\).  Only the first and third coordinates contribute to the
	difference.  The case of \(q_\varepsilon\) is identical by symmetry.
\end{proof}

\begin{lemma}[Endpoint constraints dominate the midpoint constraint]
	\label{lem:endpoint-constraints-dominate}
	Let \(\psi\in C^3(0,1)\) be concave, and assume
	\[
	\psi''(b)<0,
	\qquad
	\eta_\psi''(b)>0.
	\]
	Then, for all sufficiently small \(\varepsilon>0\),
	\begin{equation}\label{eq:endpoint-constraints-dominate}
		S_2(m_\varepsilon)<
		\min\{S_i(p_\varepsilon),S_i(q_\varepsilon):i=1,\dots,n\}.
	\end{equation}
\end{lemma}

\begin{proof}
	See Appendix \ref{app:proof-of-endpoint-constraints-dominate-the-midpoint-constraint}.
\end{proof}

\begin{lemma}[The threshold interval is nonempty and has Jensen length]
	\label{lem:shift-interval}
	Under the assumptions of Lemma \ref{lem:endpoint-constraints-dominate},
	define
	\[
	C_\varepsilon:=
	\min\left\{
	S(p_\varepsilon),S(q_\varepsilon),S(m_\varepsilon),
	\min_iS_i(p_\varepsilon),
	\min_iS_i(q_\varepsilon)
	\right\}.
	\]
	Then, for all sufficiently small \(\varepsilon>0\),
	\[
	C_\varepsilon
	=
	S_2(p_\varepsilon)=S_1(p_\varepsilon)
	=
	S_2(q_\varepsilon)=S_3(q_\varepsilon),
	\]
	and the interval
	\[
	\mathcal I_\varepsilon:=(S_2(m_\varepsilon),C_\varepsilon)
	\]
	is nonempty.  Moreover,
	\[
	|\mathcal I_\varepsilon|
	=
	J_{\eta_\psi}\left(d_\varepsilon,\frac{3\varepsilon}{2}\right)
	=
	\frac94 \eta_\psi''(b)\varepsilon^2+o(\varepsilon^2).
	\]
\end{lemma}

\begin{proof}
	See Appendix \ref{app:proof-of-nonempty-threshold-interval-lemma}.
\end{proof}

\begin{theorem}[Mechanism theorem: near-barycenter Jensen-difference nonconvexity]
	\label{thm:local-jensen-nonconvexity}
	Let \(n\ge3\), let \(\psi\in C^3(0,1)\) be concave, and assume
	\[
	\psi''(b)<0,
	\qquad
	\eta_\psi''(b)>0.
	\]
	For all sufficiently small \(\varepsilon>0\), set
	\[
	\mathcal I_\varepsilon:=(S_2(m_\varepsilon),C_\varepsilon),
	\]
	where \(C_\varepsilon\) is defined in Lemma \ref{lem:shift-interval}.
	Then, for every
	\(T\in\mathcal I_\varepsilon\), the threshold long-only region
	\(\mathcal L_T\) in \eqref{eq:threshold-long-only-region} is nonconvex.
	More precisely,
	\[
	p_\varepsilon,q_\varepsilon\in\mathcal L_T,
	\qquad
	m_\varepsilon\in D_{\Phi_{-T}}\setminus\mathcal L_T,
	\]
	where \(\Phi_{-T}(\mu)=S(\mu)-T\).  This happens although \(m_\varepsilon\) is closer to
	the barycenter than the two endpoints.  Equivalently, for vertical shifts
	\(B=-T\), every \(B\in(-C_\varepsilon,-S_2(m_\varepsilon))\) produces the
	same midpoint nonconvexity.
\end{theorem}

\begin{proof}
	Choose \(\varepsilon>0\) small enough so that Lemmas
	\ref{lem:endpoint-constraints-dominate} and \ref{lem:shift-interval}
	apply, and fix \(T\in\mathcal I_\varepsilon\).  Lemma
	\ref{lem:shift-interval} gives \(S_2(m_\varepsilon)<T<C_\varepsilon\).
	Since \(C_\varepsilon\) is bounded above by
	\(S(p_\varepsilon)\), \(S(q_\varepsilon)\), \(S(m_\varepsilon)\),
	\(\min_iS_i(p_\varepsilon)\), and \(\min_iS_i(q_\varepsilon)\), we have
	\[
	S(p_\varepsilon),S(q_\varepsilon)>T,
	\qquad
	S_i(p_\varepsilon),S_i(q_\varepsilon)\ge T
	\quad\text{for all }i.
	\]
	Thus \(p_\varepsilon,q_\varepsilon\in\mathcal L_T\).  Also
	\(S(m_\varepsilon)>T\), so \(m_\varepsilon\in D_{\Phi_{-T}}\).  However
	\(S_2(m_\varepsilon)<T\), so the second long-only constraint fails at
	\(m_\varepsilon\).  Hence
	\(m_\varepsilon\in D_{\Phi_{-T}}\setminus\mathcal L_T\).  Since
	\(m_\varepsilon=(p_\varepsilon+q_\varepsilon)/2\), the region
	\(\mathcal L_T\) is nonconvex.
\end{proof}

\begin{remark}[Interpretation of the intuition failure]
	The second coordinate is fixed along the segment:
	\[
	(p_\varepsilon)_2=(q_\varepsilon)_2=(m_\varepsilon)_2=a_\varepsilon.
	\]
	The midpoint only averages the first and third coordinates and is closer to
	\(\bar e\), but the aggregation term \(\sum_j\eta_\psi(\mu_j)\) can be
	smaller there. 
\end{remark}

\subsection{Quadratic and power-family examples}
\label{sec:quadratic-power}

The preceding theorem turns the midpoint Jensen-difference mechanism into
concrete nonconvex examples.  We record the power family as a model class, with the quadratic family as the simplest special case.  The quadratic
generator is one of the simplest non-affine symmetric separable examples,
while power-type generators are standard in stochastic portfolio theory,
notably through diversity-weighted portfolios; see
\cite{Fernholz2002SPT,FernholzKaratzasKardaras2005DiversityRelativeArbitrage}.

\begin{example}[Power family with \(r>1\) and quadratic family]
	\label{ex:power-family-nonconvex}
	Let \(n\ge3\) and \(r>1\).  Consider the symmetric separable family
	\[
	\Phi_B^{(r)}(\mu)=B-\sum_{i=1}^n\mu_i^r.
	\]
	Then there exist shifts \(B\in\mathbb R\) for which the corresponding
	long-only region is nonconvex.
	
	Indeed, take \(\psi_r(t)=-t^r\).  Then
	\[
	\psi_r''(t)=-r(r-1)t^{r-2}<0,
	\qquad
	\eta_{\psi_r}(t)
	=
	\psi_r(t)-t\psi_r'(t)
	=
	(r-1)t^r,
	\]
	and hence
	\[
	\eta_{\psi_r}''(t)=r(r-1)^2t^{r-2}>0.
	\]
	Thus Theorem \ref{thm:local-jensen-nonconvexity} applies and gives
	shifts \(B\) for which the corresponding long-only region is nonconvex.
	
	The case \(r=2\) is the quadratic family.
	In this case \(\psi(t)=-t^2\), \(\eta_\psi(t)=t^2\), and
	\(\eta_\psi''=2\), so the quadratic nonconvexity examples are recovered as
	a special case of the power family.
\end{example}

\subsubsection{An explicit three-dimensional quadratic witness}
\label{sec:sym-counterexample}

Example \ref{ex:power-family-nonconvex} already gives nonconvex long-only
regions for the quadratic family in every dimension \(n\ge3\).  We record the
three-dimensional case because it is the simplest explicit instance of the
general mechanism and serves as a convenient visual reference.

\begin{example}[A symmetric long-short portfolio with nonconvex long-only region]
	\label{cex:sym-quadratic}
	Let
	\[
	\Phi(p)=\frac{33}{80}-\sum_{i=1}^3p_i^2.
	\]
	Its positive region is
	\[
	D_\Phi
	=
	\left\{p\in\Delta^{(3)}:\sum_{i=1}^3p_i^2<\frac{33}{80}\right\}.
	\]
	Then \(\Phi\) is positive, symmetric and concave on \(D_\Phi\), and its
	long-only region is not convex.
\end{example}

\begin{proof}
	For this generator,
	\[
	G_i(p)=\frac{33}{80}+\sum_{j=1}^3p_j^2-2p_i.
	\]
	Take
	\[
	p=\left(\frac38,\frac38,\frac14\right),
	\qquad
	q=\left(\frac14,\frac38,\frac38\right),
	\qquad
	m=\frac{p+q}{2}=\left(\frac5{16},\frac38,\frac5{16}\right).
	\]
	A direct calculation gives
	\[
	G_2(p)=G_2(q)=\frac1{160},
	\qquad
	G_2(m)=-\frac1{640}.
	\]
	Moreover, one checks that \(\Phi\), \(G_1\), and \(G_3\) are positive at
	all three points \(p,q,m\).  Hence \(p,q\in\mathcal L\), while
	\(m\in D_\Phi\setminus\mathcal L\).  Since \(m=(p+q)/2\), the long-only
	region is not convex.
\end{proof}

\begin{remark}[Geometric form of the constraints]
	Since
	\[
	G_i(p)=\frac{33}{80}+\sum_{j=1}^3p_j^2-2p_i
	=\|p-e_i\|^2-\frac{47}{80},
	\]
	the boundary \(G_i=0\) is a circle, in the affine plane of the simplex,
	centered at the vertex \(e_i\) with radius \(\sqrt{47/80}\).  Thus the
	short-selling region is governed by vertex-centered constraints rather than
	simply by distance from the barycenter.
\end{remark}

Figure \ref{fig:sym-counterexample-visual-summary} shows the positive region,
the constraint \(G_2\ge0\), and the long-only region in the projected
\((p_1,p_2)\)-plane.

\begin{figure}[H]
	\centering
	\begin{subfigure}[b]{0.32\textwidth}
		\centering
		\includegraphics[width=\textwidth]{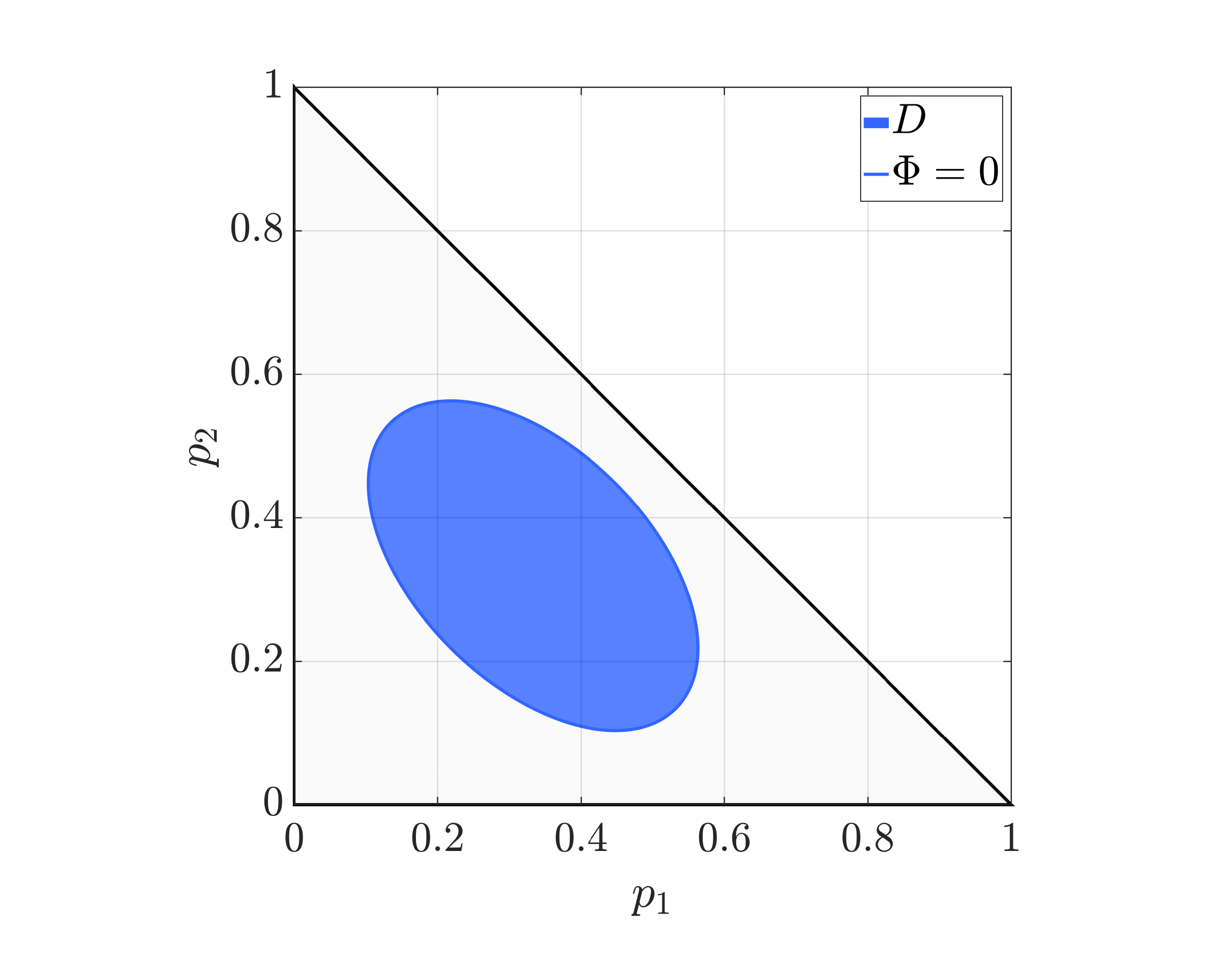}
		\caption{Positive region \(D_\Phi\)}
	\end{subfigure}
	\hfill
	\begin{subfigure}[b]{0.32\textwidth}
		\centering
		\includegraphics[width=\textwidth]{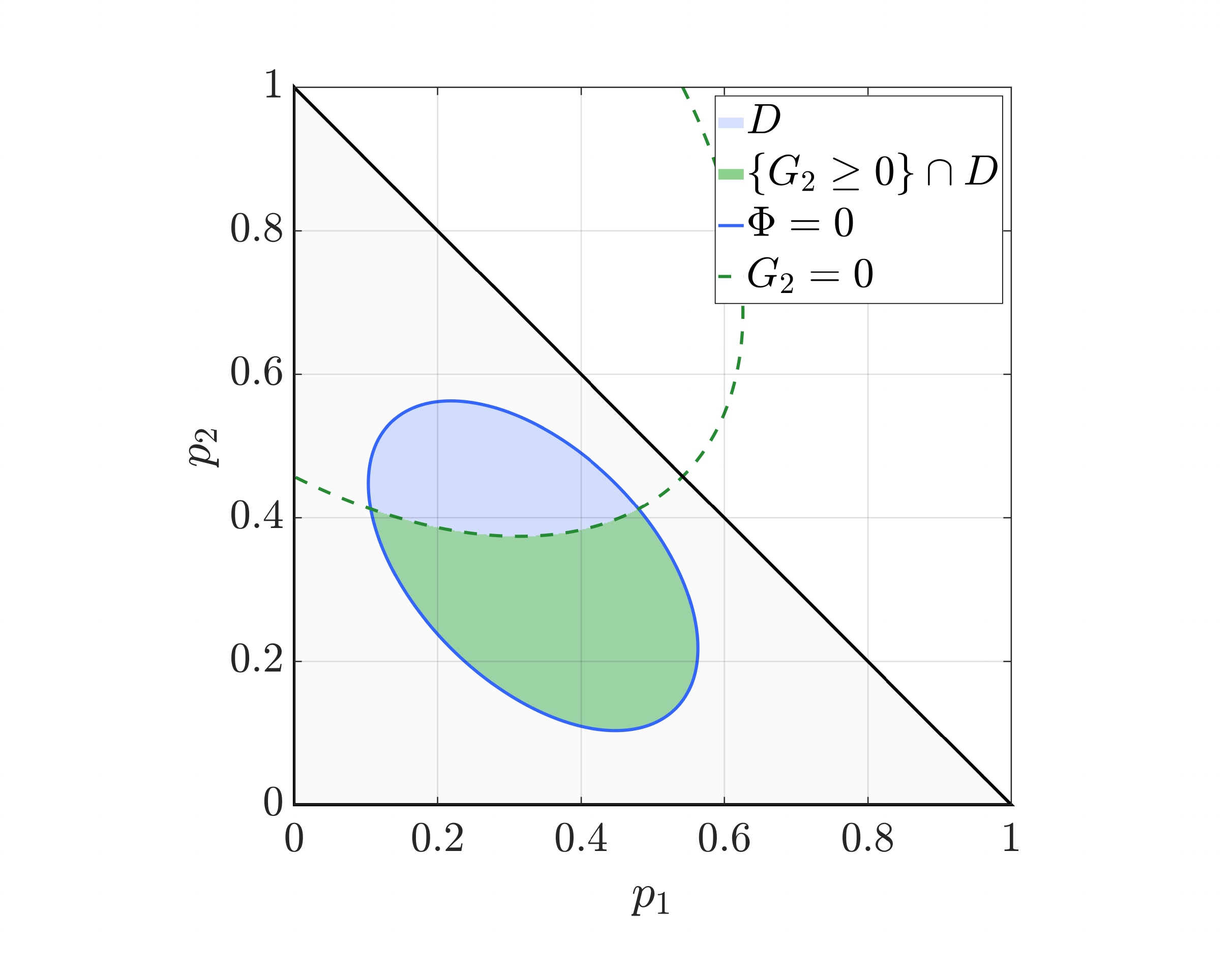}
		\caption{The constraint \(G_2\ge0\)}
	\end{subfigure}
	\hfill
	\begin{subfigure}[b]{0.32\textwidth}
		\centering
		\includegraphics[width=\textwidth]{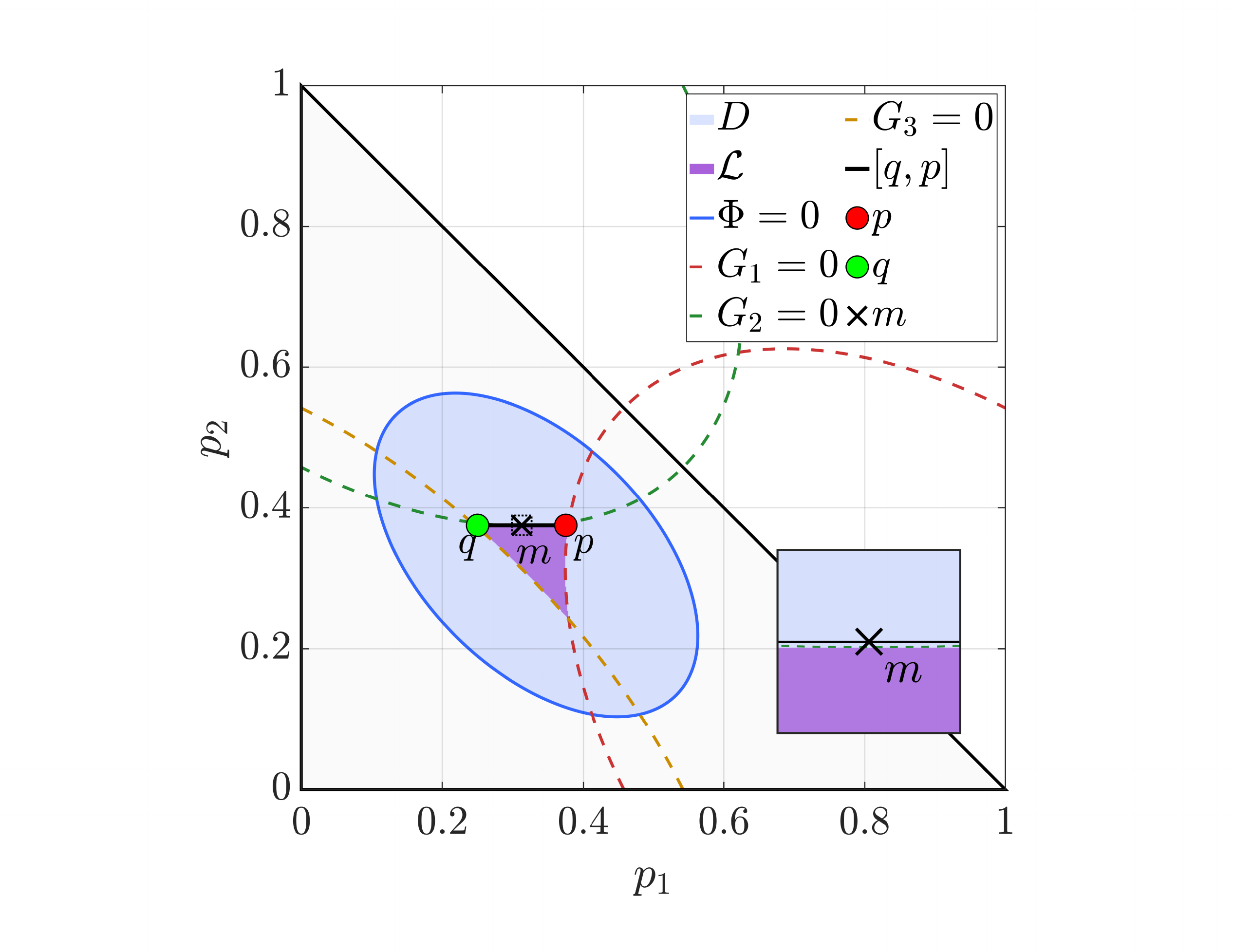}
		\caption{Long-only region}
	\end{subfigure}
	\caption{The three-dimensional quadratic counterexample.  The endpoints
		are long-only, while their midpoint fails the second long-only
		constraint.}
	\label{fig:sym-counterexample-visual-summary}
\end{figure}
\section{Barycentric stability: criterion and failure}
\label{sec:star-shaped}

The previous section shows that convexity of \(\eta_\psi\) near the barycenter can produce
chords whose endpoints are long-only while their midpoint is not.  Such a
midpoint construction tests convexity along arbitrary chords, but it does not
address the distinguished direction toward the barycenter.  We therefore
study barycentric star-shapedness, the center-directed stability condition
appropriate for symmetric problems on the simplex. 

\subsection{The threshold criterion}
\label{subsec:minimal-constraint-threshold}

The next result gives a necessary and sufficient threshold inequality for
star-shapedness with respect to \(\bar e\).  The parameter \(T\) is only the
vertical shift of the same symmetric separable shape.  Requiring the condition
for all \(T\) separates the geometry of the generator from a particular
normalization.

\begin{theorem}[Criterion theorem: barycentric threshold criterion]
	\label{thm:sharp-barycentric-threshold}
	Let \(n\ge2\), and let \(\psi\in C^2(0,1)\) be concave.  With
	\(\Theta_\psi\) as in Definition \ref{def:minimal-constraint-function} and
	\(\mathcal L_T\) as in \eqref{eq:LT-active-form}, the following are
	equivalent:
	\begin{enumerate}
		\item For every \(T\in\mathbb R\), the region \(\mathcal L_T\) is
		star-shaped with respect to \(\bar e\), whenever it is nonempty.
		\item For every \(\mu\in\Delta^{(n)}\) and \(\lambda\in[0,1]\),
		\begin{equation}\label{eq:sharp-threshold-inequality}
			\Theta_\psi(\mu^\lambda)
			\ge
			\min\{S(\mu),\Theta_\psi(\mu)\},
			\qquad
			\mu^\lambda:=(1-\lambda)\bar e+\lambda\mu.
		\end{equation}
	\end{enumerate}
\end{theorem}

\begin{proof}
	First, \(S\) does not decrease under barycentric contraction.  Indeed, by
	concavity of \(\psi\),
	\[
	S(\mu^\lambda)
	\ge
	(1-\lambda)n\psi(1/n)+\lambda S(\mu)
	\ge S(\mu),
	\]
	where the last inequality is Jensen's inequality \(S(\mu)\le n\psi(1/n)\).
	
	Assume first that every nonempty \(\mathcal L_T\) is star-shaped.  If
	\eqref{eq:sharp-threshold-inequality} failed for some \(\mu,\lambda\), one
	could choose \(T\) such that
	\[
	\Theta_\psi(\mu^\lambda)<T<\min\{S(\mu),\Theta_\psi(\mu)\}.
	\]
	Then \(\mu\in\mathcal L_T\), but star-shapedness would force
	\(\mu^\lambda\in\mathcal L_T\), contradicting
	\(\Theta_\psi(\mu^\lambda)<T\).
	
	Conversely, assume \eqref{eq:sharp-threshold-inequality}.  If
	\(\mu\in\mathcal L_T\), then \(S(\mu)>T\) and \(\Theta_\psi(\mu)\ge T\).
	Hence \(S(\mu^\lambda)\ge S(\mu)>T\), and
	\[
	\Theta_\psi(\mu^\lambda)
	\ge
	\min\{S(\mu),\Theta_\psi(\mu)\}
	\ge T.
	\]
	Thus \(\mu^\lambda\in\mathcal L_T\).
\end{proof}

\begin{remark}[Role of the criterion]
	The criterion is not meant to be the final recovery result by itself.  Its
	role is to remove the threshold parameter from the geometry and reduce
	uniform barycentric stability to a pointwise inequality for
	\(\Theta_\psi\).  The quantitative deficit index and the recovery theorems
	below are the main uses of this reduction.
\end{remark}

\begin{definition}[Barycentric threshold deficit index]
	\label{def:barycentric-deficit-index}
	For \(\mu\in\Delta^{(n)}\) and \(\lambda\in[0,1]\), define
	\[
	d_\psi(\mu,\lambda)
	:=
	\left[
	\min\{S(\mu),\Theta_\psi(\mu)\}
	-
	\Theta_\psi(\mu^\lambda)
	\right]_+,
	\qquad
	\mu^\lambda=(1-\lambda)\bar e+\lambda\mu.
	\]
	The global barycentric threshold deficit is
	\[
	\mathfrak D_\psi
	:=
	\sup_{\mu\in\Delta^{(n)},\ 0\le\lambda\le1}
	d_\psi(\mu,\lambda).
	\]
\end{definition}

\begin{corollary}[Deficit index and uniform barycentric star-shapedness]
	\label{cor:deficit-index}
	With the notation above, \(\mathfrak D_\psi=0\) if and only if every
	nonempty \(\mathcal L_T\), \(T\in\mathbb R\), is star-shaped with respect
	to \(\bar e\).
\end{corollary}

\begin{proof}
	This is precisely the threshold criterion in Theorem
	\ref{thm:sharp-barycentric-threshold}, written in terms of the positive
	part of the threshold deficit.
\end{proof}

\begin{remark}[Interpretation of the deficit index]
	The index \(\mathfrak D_\psi\) measures the largest threshold interval
	length over which barycentric star-shapedness can fail.  More precisely, if
	\(d_\psi(\mu,\lambda)>0\), then every threshold
	\[
	T\in
	\bigl(\Theta_\psi(\mu^\lambda),\min\{S(\mu),\Theta_\psi(\mu)\}\bigr)
	\]
	satisfies \(\mu\in\mathcal L_T\) but
	\(\mu^\lambda\notin\mathcal L_T\), and the length of this interval is
	\(d_\psi(\mu,\lambda)\).  Thus \(\mathfrak D_\psi=0\) is the necessary and sufficient
	condition for uniform barycentric star-shapedness over all vertical shifts,
	while a positive value records the size of the worst barycentric threshold
	gap.
\end{remark}

\subsection{A long-only region which is not barycentrically star-shaped}
\label{subsec:nonstar-barycenter-example}

The preceding criterion is sharp, and the required inequality may fail even
within the symmetric separable class.  The following example gives a direct
failure of barycentric star-shapedness.

\begin{example}[Failure of barycentric star-shapedness]
	\label{ex:nonstar-barycenter}
	Work on \(\Delta^{(4)}\) with barycenter \(\bar e\), and let
	\[
	\psi(t)=-\frac{(1-t)^5}{20},
	\qquad 0<t<1.
	\]
	Then \(\psi'(t)=(1-t)^4/4\), \(\psi''(t)=-(1-t)^3<0\), and
	\[
	\eta_\psi(t)=\psi(t)-t\psi'(t)
	=-\frac{(1-t)^4(1+4t)}{20}.
	\]
	Take
	\[
	p=(0.6,0.399,0.0005,0.0005),
	\qquad
	\lambda=0.98,
	\qquad
	p^\lambda=(1-\lambda)\bar e+\lambda p,
	\]
	so \(p^\lambda=(0.593,0.39602,0.00549,0.00549)\).  For
	\(T=-0.11491\), direct calculation gives
	\[
	\begin{array}{c|cccc}
		\mu & S(\mu) & S(\mu)-T & \Theta_\psi(\mu) & \Theta_\psi(\mu)-T \\
		\hline
		\bar e & -0.047461 & 0.067449 & -0.047461 & 0.067449 \\
		p & -0.104183 & 0.010727 & -0.114886 & 0.000024 \\
		p^\lambda & -0.101862 & 0.013048 & -0.114930 & -0.000020
	\end{array}
	\]
	At the barycenter, \(S(\bar e)=\Theta_\psi(\bar e)\) because all
	coordinates are equal.  Thus \(\bar e,p\in\mathcal L_T\), while
	\(p^\lambda\in D_{\Phi_{-T}}\setminus\mathcal L_T\).  Since
	\(p^\lambda\) lies on the segment from \(\bar e\) to \(p\), the region
	\(\mathcal L_T\) is not star-shaped with respect to the barycenter.
\end{example}

To visualize the example, consider the slice
\[\Sigma=\{(x,y,z,z)\in\Delta^{(4)}:z=(1-x-y)/2\},\]
which contains \(\bar e\), \(p\), and \(p^\lambda\).
Figure~\ref{fig:nonstar-figure-group} shows the positive region, the
corresponding long-only region on this slice, and a zoom near \(p^\lambda\).

\begin{figure}[H]
	\centering
	\begin{subfigure}[b]{0.32\textwidth}
		\centering
		\includegraphics[width=\textwidth]{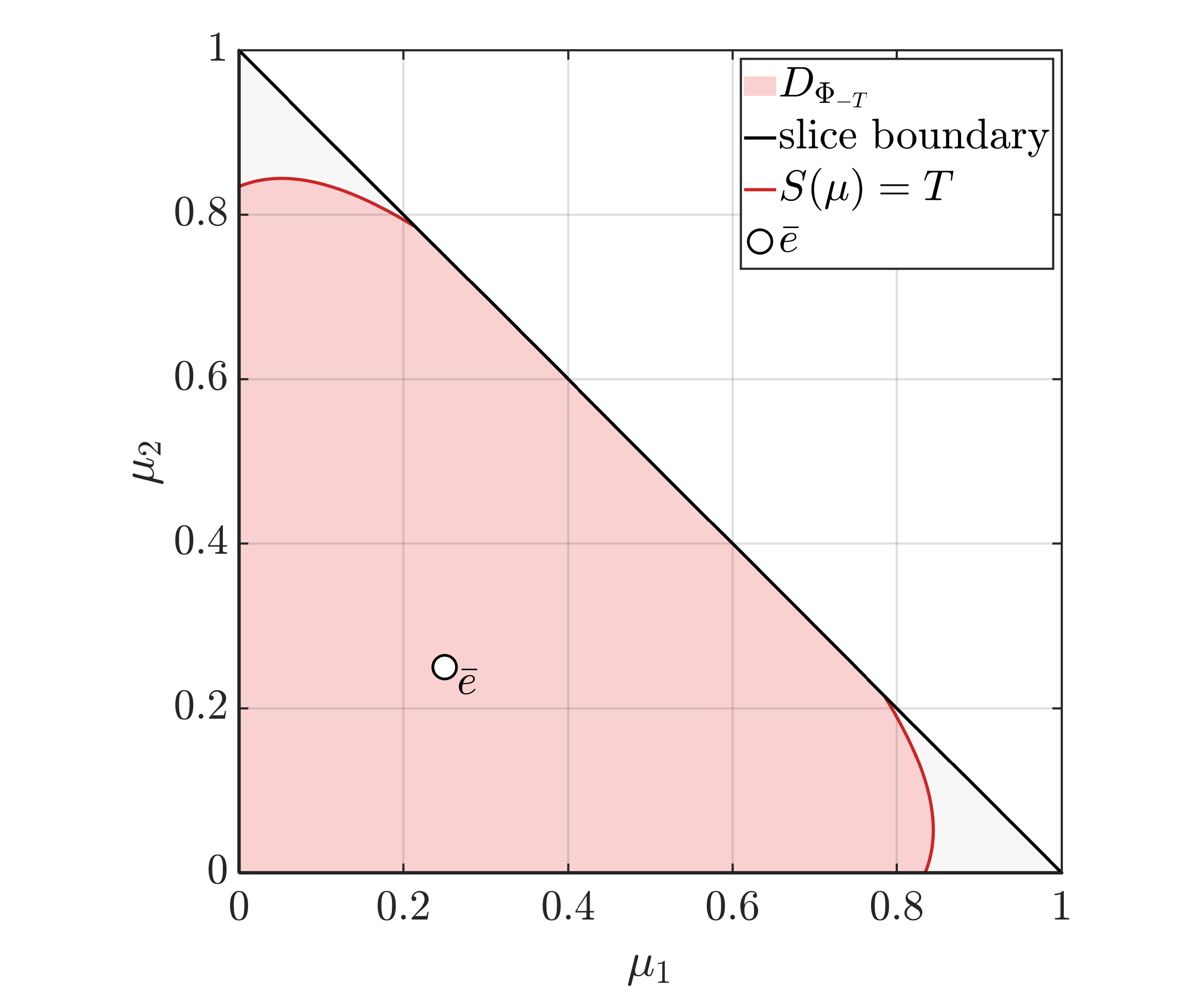}
		\caption{\(D_{\Phi_{-T}}\cap\Sigma\).}
		\label{fig:nonstar-domain-slice}
	\end{subfigure}
	\hfill
	\begin{subfigure}[b]{0.32\textwidth}
		\centering
		\includegraphics[width=\textwidth]{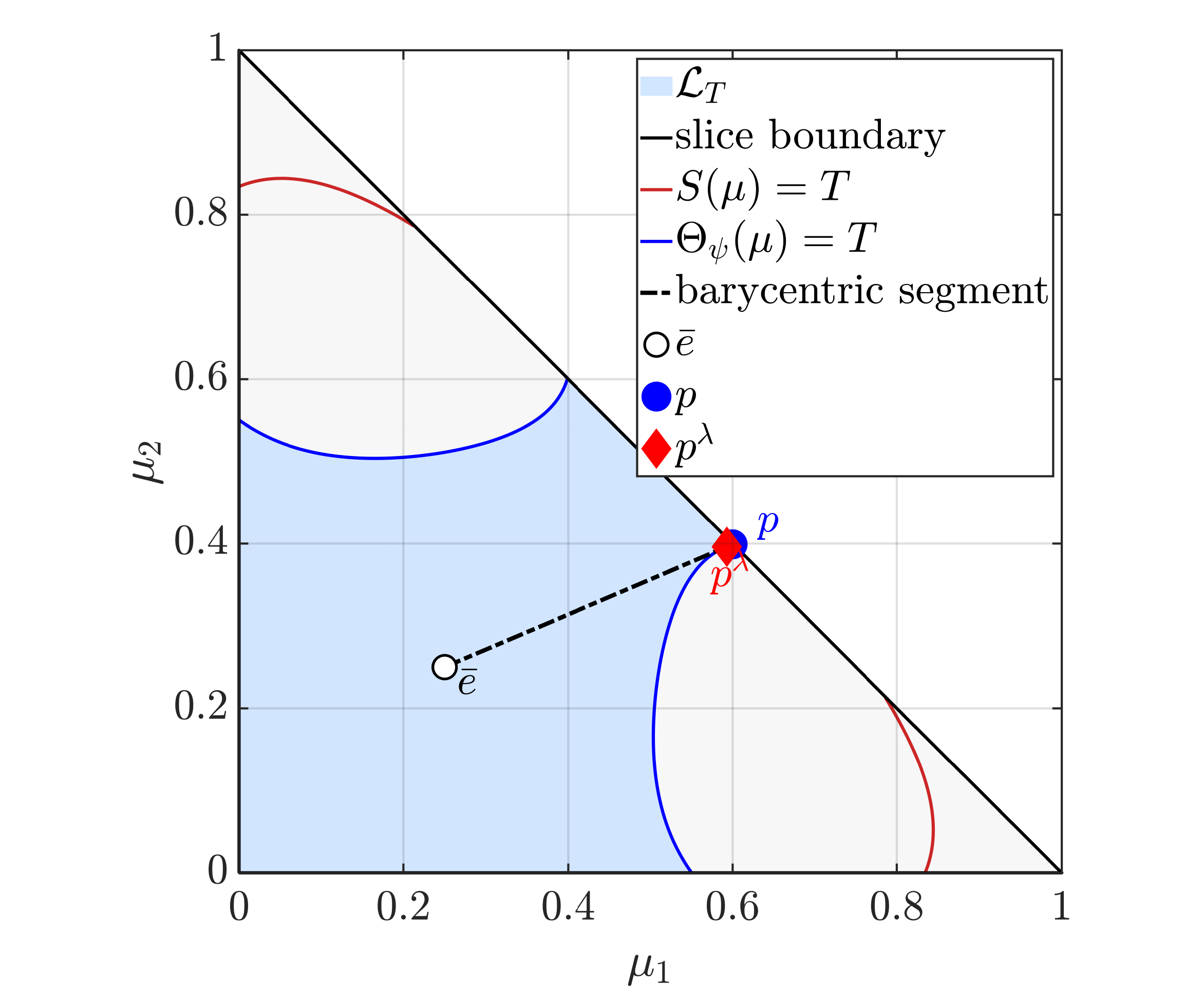}
		\caption{\(\mathcal L_T\cap\Sigma\).}
		\label{fig:nonstar-global-slice}
	\end{subfigure}
	\hfill
	\begin{subfigure}[b]{0.32\textwidth}
		\centering
		\includegraphics[width=\textwidth]{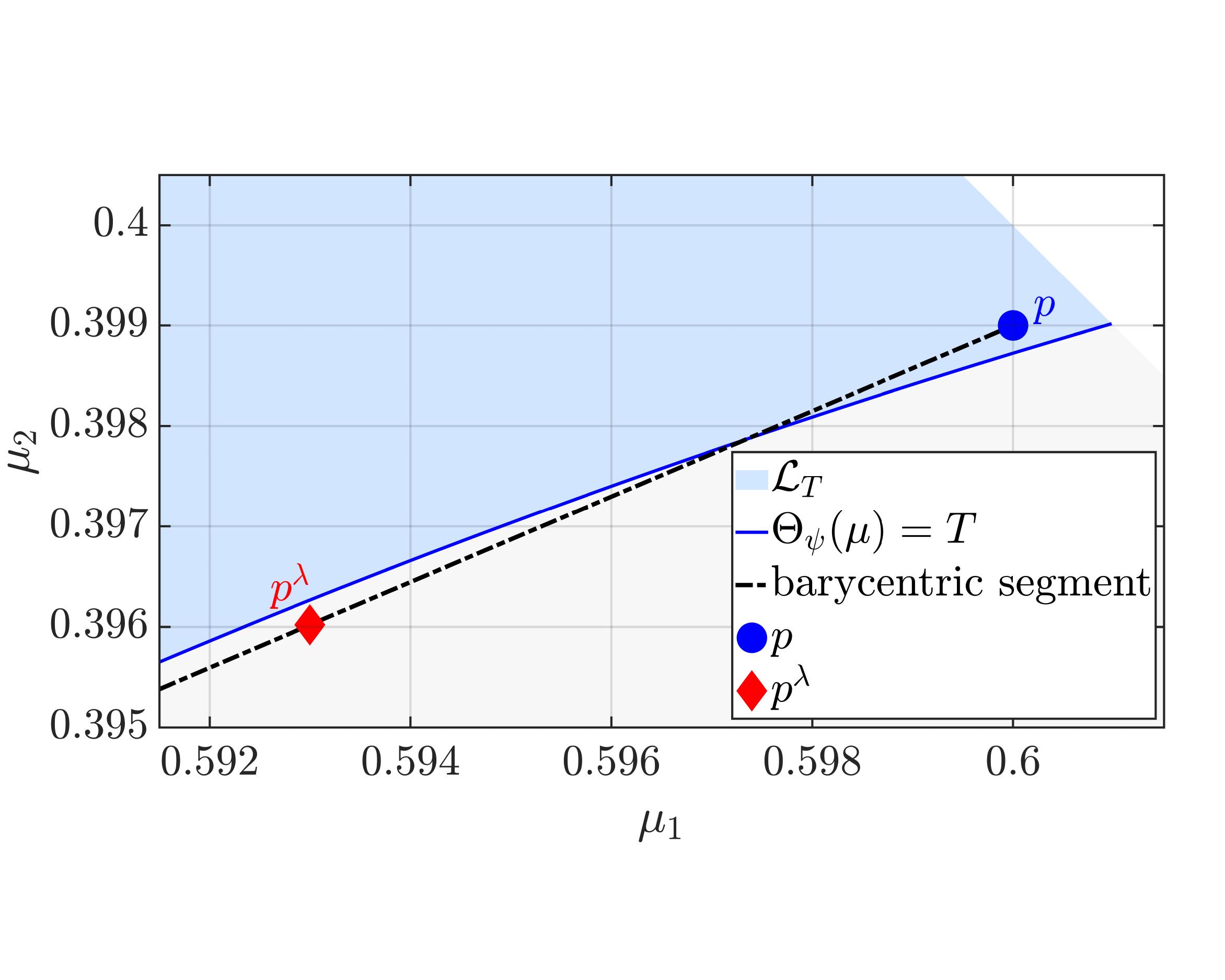}
		\caption{Zoom near \(p^\lambda\).}
		\label{fig:nonstar-zoom}
	\end{subfigure}
	\caption{Failure of barycentric star-shapedness on the slice \(\Sigma\).
		The first panel shows the positive region of the shifted generator.
		The point \(p\) is long-only, while the intermediate point \(p^\lambda\)
		lies in \(D_{\Phi_{-T}}\setminus\mathcal L_T\).}
	\label{fig:nonstar-figure-group}
\end{figure}
\section{Recovery of barycentric stability}
\label{sec:recovery-of-barycentric-stability}
The preceding section shows that barycentric star-shapedness can fail.  We now
give two sufficient conditions that recover it: concavity of the aggregation
function and second-derivative domination by the coordinate term.

\subsection{Sufficient conditions: concavity and second-derivative domination}
\label{subsec:sufficient-conditions-star-shapedness}

The first recovery result shows that concavity of the aggregation function
\(\eta_\psi\) is sufficient to restore barycentric star-shapedness.

\begin{lemma}[Concavity of \(\eta_\psi\) implies barycentric monotonicity]
	\label{lem:concave-eta-implies-theta}
	If \(\eta_\psi\) is concave on \((0,1)\), then
	\[
	\Theta_\psi(\mu^\lambda)\ge\Theta_\psi(\mu),
	\qquad
	\mu^\lambda=(1-\lambda)\bar e+\lambda\mu.
	\]
\end{lemma}

\begin{proof}
	By the same Jensen argument used for \(S\) in Theorem
	\ref{thm:sharp-barycentric-threshold},
	\[
	\sum_i\eta_\psi(\mu_i^\lambda)\ge\sum_i\eta_\psi(\mu_i).
	\]
	Also
	\[
	(\mu^\lambda)_{(1)}=(1-\lambda)/n+\lambda\mu_{(1)}\le\mu_{(1)},
	\]
	and \(\psi'\) is nonincreasing.  Adding these inequalities gives the claim.
\end{proof}

\begin{theorem}[Recovery theorem under concave aggregation]
	\label{thm:star-shaped-concave-eta}
	Let \(n\ge2\).  If \(\psi\) is concave and
	\(\eta_\psi(t)=\psi(t)-t\psi'(t)\) is concave on
	\((0,1)\), then every nonempty \(\mathcal L_T\) is star-shaped with respect
	to \(\bar e\).
\end{theorem}

\begin{proof}
	By Lemma \ref{lem:concave-eta-implies-theta},
	\(\Theta_\psi(\mu^\lambda)\ge\Theta_\psi(\mu)\), hence
	\eqref{eq:sharp-threshold-inequality} holds.  The result follows from
	Theorem \ref{thm:sharp-barycentric-threshold}.
\end{proof}

\begin{remark}[The sign of the Jensen difference]
	Convexity of \(\eta_\psi\) may lower the aggregation term under averaging
	and create midpoint failures.  Concavity of \(\eta_\psi\) has the opposite
	effect along barycentric segments.
\end{remark}

We next replace concavity of \(\eta_\psi\) by a second-derivative domination
condition.  Positive second derivative of \(\eta_\psi\) is allowed, provided
it is controlled by the concavity of \(\psi\).

\begin{lemma}[Barycentric difference identities]
	\label{lem:barycentric-difference-identities}
	Assume that \(\psi\) and \(\eta_\psi\) are \(C^2\) on \((0,1)\), where
	\(\eta_\psi(t)=\psi(t)-t\psi'(t)\).  For \(\mu\in\Delta^{(n)}\), set
	\[
	b=\frac1n,
	\qquad
	\delta_i=\mu_i-b,
	\qquad
	\delta_*:=\mu_{(1)}-b,
	\qquad
	\mu^\lambda=(1-\lambda)\bar e+\lambda\mu.
	\]
	Then
	\begin{equation}\label{eq:exact-S-barycentric-identity}
		S(\mu^\lambda)-S(\mu)
		=
		\int_\lambda^1\int_0^s
		\sum_{i=1}^n\delta_i^2[-\psi''(b+u\delta_i)]\,du\,ds,
	\end{equation}
	and
	\begin{equation}\label{eq:exact-Theta-barycentric-identity}
		\begin{aligned}
			\Theta_\psi(\mu^\lambda)-\Theta_\psi(\mu)
			&=
			\int_\lambda^1\int_0^s
			\sum_{i=1}^n\delta_i^2[-\eta_\psi''(b+u\delta_i)]\,du\,ds  \\
			&\quad+
			\int_\lambda^1\delta_*[-\psi''(b+s\delta_*)]\,ds .
		\end{aligned}
	\end{equation}
	In particular, if \(\psi\) and \(\eta_\psi\) are concave, then both \(S\)
	and the minimal constraint function \(\Theta_\psi\) do not decrease under
	barycentric contraction.
\end{lemma}

\begin{proof}
	See Appendix \ref{app:proof-barycentric-difference-identities}.
\end{proof}

\begin{theorem}[Recovery theorem: second-derivative domination and barycentric threshold deficit]
	\label{thm:curvature-domination-barycentric}
	Let \(n\ge2\).  Let \(\psi\) and \(\eta_\psi\) be \(C^2\) on \((0,1)\), and suppose that
	\(\psi\) is concave.  Assume the following second-derivative bounds: for
	some \(K_\eta,\kappa_\psi\ge0\),
	\[
	\eta_\psi''(t)\le K_\eta,
	\qquad
	-\psi''(t)\ge\kappa_\psi,
	\qquad t\in(0,1).
	\]
	Then, for every \(\mu\in\Delta^{(n)}\) and \(\lambda\in[0,1]\),
	\begin{equation}\label{eq:S-barycentric-lower-bound}
		S(\mu^\lambda)-S(\mu)
		\ge
		\frac{\kappa_\psi}{2}(1-\lambda^2)\|\mu-\bar e\|^2,
	\end{equation}
	and
	\begin{equation}\label{eq:curvature-domination-estimate}
		\Theta_\psi(\mu^\lambda)-\Theta_\psi(\mu)
		\ge
		\left(\kappa_\psi-\frac{1+\lambda}{2}K_\eta\right)
		(1-\lambda)
		\left(\mu_{(1)}-\frac1n\right).
	\end{equation}
	Consequently,
	\begin{equation}\label{eq:deficit-pointwise-bound}
		d_\psi(\mu,\lambda)
		\le
		(1-\lambda)
		\left[
		\frac{1+\lambda}{2}K_\eta-\kappa_\psi
		\right]_+
		\left(\mu_{(1)}-\frac1n\right),
	\end{equation}
	and
	\begin{equation}\label{eq:deficit-global-bound}
		\mathfrak D_\psi
		\le
		\left(1-\frac1n\right)
		\sup_{0\le\lambda\le1}
		(1-\lambda)
		\left[
		\frac{1+\lambda}{2}K_\eta-\kappa_\psi
		\right]_+.
	\end{equation}
	If \(K_\eta>0\), then
	\[
	\mathfrak D_\psi
	\le
	\left(1-\frac1n\right)
	\frac{(K_\eta-\kappa_\psi)_+^2}{2K_\eta},
	\]
	while if \(K_\eta=0\), then \(\mathfrak D_\psi=0\).  In particular, if
	\(\kappa_\psi\ge K_\eta\), then every nonempty \(\mathcal L_T\) is
	star-shaped with respect to \(\bar e\).
	
	Moreover, if \(-\eta_\psi''(t)\ge\kappa_\eta\ge0\) on \((0,1)\), then
	\begin{equation}\label{eq:Theta-quantitative-lower-bound}
		\Theta_\psi(\mu^\lambda)-\Theta_\psi(\mu)
		\ge
		\frac{\kappa_\eta}{2}(1-\lambda^2)\|\mu-\bar e\|^2
		+
		\kappa_\psi(1-\lambda)\left(\mu_{(1)}-\frac1n\right).
	\end{equation}
\end{theorem}

\begin{proof}
	Set \(b=1/n\), \(\delta_i=\mu_i-b\), and
	\(\delta_*=\mu_{(1)}-1/n\).  By Lemma
	\ref{lem:barycentric-difference-identities} and
	\(-\psi''\ge\kappa_\psi\),
	\[
	S(\mu^\lambda)-S(\mu)
	\ge
	\kappa_\psi
	\int_\lambda^1\int_0^s\sum_i\delta_i^2\,du\,ds
	=
	\frac{\kappa_\psi}{2}(1-\lambda^2)\|\mu-\bar e\|^2,
	\]
	which proves \eqref{eq:S-barycentric-lower-bound}.

	The same lemma, together with
	\(\eta_\psi''\le K_\eta\) and \(-\psi''\ge\kappa_\psi\), gives
	\[
	\Theta_\psi(\mu^\lambda)-\Theta_\psi(\mu)
	\ge
	-\frac{K_\eta}{2}(1-\lambda^2)\|\mu-\bar e\|^2
	+
	\kappa_\psi(1-\lambda)\delta_* .
	\]
	Using
	\[
	\|\mu-\bar e\|^2
	=\sum_i\mu_i^2-\frac1n
	\le \mu_{(1)}-\frac1n
	=\delta_*,
	\]
	and \(1-\lambda^2=(1-\lambda)(1+\lambda)\), we obtain
	\eqref{eq:curvature-domination-estimate}.

	Moreover,
	\[
	d_\psi(\mu,\lambda)
	\le[\Theta_\psi(\mu)-\Theta_\psi(\mu^\lambda)]_+,
	\]
	so \eqref{eq:deficit-pointwise-bound} follows from
	\eqref{eq:curvature-domination-estimate}.  Taking the supremum and using
	\(0\le\mu_{(1)}-1/n\le1-1/n\) gives
	\eqref{eq:deficit-global-bound}.  The remaining scalar supremum is
	\[
	\sup_{0\le\lambda\le1}(1-\lambda)
	\left[\frac{1+\lambda}{2}K_\eta-\kappa_\psi\right]_+
	=
	\begin{cases}
	0, & K_\eta=0,\\[2mm]
	\dfrac{(K_\eta-\kappa_\psi)_+^2}{2K_\eta}, & K_\eta>0,
	\end{cases}
	\]
	where the nonzero case follows by maximizing the quadratic
	\((1-\lambda)((1+\lambda)K_\eta/2-\kappa_\psi)\) at
	\(\lambda=\kappa_\psi/K_\eta\) when \(K_\eta>\kappa_\psi\).

	If \(\kappa_\psi\ge K_\eta\), then the coefficient in
	\eqref{eq:curvature-domination-estimate} is nonnegative for all
	\(\lambda\), so \(\Theta_\psi(\mu^\lambda)\ge\Theta_\psi(\mu)\).  The
	barycentric threshold criterion then gives star-shapedness of every
	nonempty \(\mathcal L_T\).

	Finally, if \(-\eta_\psi''\ge\kappa_\eta\), applying this bound and
	\(-\psi''\ge\kappa_\psi\) directly in
	\eqref{eq:exact-Theta-barycentric-identity} yields
	\eqref{eq:Theta-quantitative-lower-bound}.
\end{proof}
\begin{remark}[Beyond concavity of \(\eta_\psi\)]
	The case \(\eta_\psi''\le0\) corresponds to \(K_\eta=0\).  The preceding
	theorem also allows \(\eta_\psi\) to have positive second derivative,
	provided that the resulting loss in the aggregation term is dominated by
	the concavity of \(\psi\) in the largest-coordinate term.
\end{remark}

\begin{corollary}[Concave power family: quantitative barycentric inequality]
	\label{cor:quantitative-power-barycentric}
	Let \(n\ge2\), \(0<r<1\), and \(\psi_r(t)=t^r\).  Then
	\[
	\Theta_{\psi_r}(\mu^\lambda)-\Theta_{\psi_r}(\mu)
	\ge
	\frac{r(1-r)^2}{2}(1-\lambda^2)\|\mu-\bar e\|^2
	+
	r(1-r)(1-\lambda)\left(\mu_{(1)}-\frac1n\right),
	\]
	where
	\[
	\Theta_{\psi_r}(\mu)=(1-r)\sum_{i=1}^n\mu_i^r+r\mu_{(1)}^{r-1}.
	\]
	Consequently, for every shift, the corresponding nonempty long-only region
	is star-shaped with respect to \(\bar e\).
\end{corollary}

\begin{proof}
	Here \(\eta_{\psi_r}(t)=(1-r)t^r\), and
	\[
	-\psi_r''(t)=r(1-r)t^{r-2}\ge r(1-r),
	\qquad
	-\eta_{\psi_r}''(t)=r(1-r)^2t^{r-2}\ge r(1-r)^2.
	\]
	Moreover, \(\eta_{\psi_r}''(t)\le0\).  Therefore Theorem
	\ref{thm:curvature-domination-barycentric}, applied with
	\[
	\kappa_\psi=r(1-r),
	\qquad
	\kappa_\eta=r(1-r)^2,
	\qquad
	K_\eta=0,
	\]
	gives the displayed inequality and shows that every nonempty shifted
	long-only region is star-shaped with respect to \(\bar e\).
\end{proof}

\subsection{Nonconvex but barycentrically star-shaped regions}
\label{subsec:nonconvex-but-star-shaped}

The preceding theorem gives sufficient conditions for barycentric
star-shapedness.  Combining it with the midpoint nonconvexity mechanism from
Section \ref{sec:intuition} shows that convexity and barycentric
star-shapedness are distinct requirements. Thus the failure of convexity does not by itself imply the failure of
center-directed stability.

\begin{corollary}[Separation consequence: nonconvex but barycentrically star-shaped regions]
	\label{cor:nonconvex-but-star-shaped}
	Let \(n\ge3\), let \(\psi\in C^3(0,1)\) be concave, and set \(b=1/n\).
	Assume that
	\[
	\psi''(b)<0,
	\qquad
	\eta_\psi''(b)>0,
	\]
	and that there exists \(K\ge0\) such that
	\[
	\eta_\psi''(t)\le K\le -\psi''(t),
	\qquad t\in(0,1).
	\]
	Then, for all sufficiently small \(\varepsilon>0\), there is a nonempty
	threshold interval \(\mathcal I_\varepsilon\) such that every
	\(\mathcal L_T\), \(T\in\mathcal I_\varepsilon\), is nonconvex, while every
	nonempty \(\mathcal L_T\) is star-shaped with respect to \(\bar e\).
\end{corollary}

\begin{proof}
	The nonconvexity interval is given by Theorem
	\ref{thm:local-jensen-nonconvexity}.  The uniform star-shapedness follows
	from Theorem \ref{thm:curvature-domination-barycentric} with
	\(K_\eta=\kappa_\psi=K\), which gives \(\mathfrak D_\psi=0\).
\end{proof}

\begin{remark}[The quadratic family]
	For \(\psi(t)=-t^2\), one has \(\eta_\psi(t)=t^2\),
	\(\psi''=-2\), and \(\eta_\psi''=2\).  Hence the assumptions of
	Corollary \ref{cor:nonconvex-but-star-shaped} hold with \(K=2\).  Together
	with Example \ref{ex:power-family-nonconvex}, this shows that the quadratic
	family can produce nonconvex long-only regions, while every nonempty
	shifted quadratic long-only region is still barycentrically star-shaped.
\end{remark}

\section{The affine aggregation case: entropy and critical power limits}
\label{sec:entropy}
We finally record the affine borderline of the aggregation function.  For
entropy, the aggregation term is constant on the simplex, and the long-only
constraints reduce to coordinate thresholds.  The entropy-normalized power
family then gives a short comparison with the concave and convex aggregation
cases.

\subsection{Shifted entropy and coordinate thresholds}
\label{subsec:shifted-entropy-coordinate-thresholds}

The Shannon entropy
\[
H(\mu):=-\sum_{i=1}^n\mu_i\log\mu_i
\]
is the canonical entropy functional on the probability simplex
\cite{Shannon1948MathematicalTheory,CoverThomas2006ElementsInformationTheory}. Related convexity and diversity viewpoints on entropy may be found in
\cite{Rao1984ConvexityEntropyDiversity}.
In the present setting, entropy is distinguished by the fact that the
aggregation function is affine, and hence the aggregation term is constant on
the simplex.

For \(0<T<\log n\), define
\[
\Phi_{-T}^H(\mu):=H(\mu)-T,
\qquad
D_{\Phi_{-T}^H}:=\{\mu\in\Delta^{(n)}:H(\mu)>T\}.
\]

\begin{proposition}[Shifted entropy generates a convex nonempty long-only region]
	\label{prop:shifted-entropy}
	Let \(n\ge2\) and \(0<T<\log n\).  Then \(\Phi_{-T}^H\) is positive, symmetric
	and concave on the open convex set \(D_{\Phi_{-T}^H}\).  Its long-only region is
	\[
	\mathcal L_T
	=
	D_{\Phi_{-T}^H}\cap\{\mu_i\le e^{-T},\ i=1,\dots,n\}.
	\]
	Hence \(\mathcal L_T\) is nonempty and convex.  Moreover, the generated
	weights are not long-only on all of \(D_{\Phi_{-T}^H}\).
\end{proposition}

\begin{proof}
	See Appendix \ref{app:proof-convex-entropy}.
\end{proof}

\begin{remark}[Entropy as the affine aggregation-function case]
	For \(\psi(t)=-t\log t\), one has \(\psi'(t)=-1-\log t\) and
	\[
	\eta_\psi(t)=\psi(t)-t\psi'(t)=t.
	\]
	Thus \(\sum_j\eta_\psi(\mu_j)=1\) on the simplex, so the aggregation term is
	constant and only shifts the constraints by the same amount for every
	coordinate.  This is why the shifted entropy family reduces to coordinate
	thresholds.
\end{remark}

\subsection{Entropy-normalized power family and critical Jensen limits}
\label{subsec:entropy-normalized-power}

We use a standard affine normalization of power sums, with Shannon entropy
recovered as \(r\to1\)
\cite{Daroczy1970GeneralizedInformationFunctions,Tsallis1988GeneralizationBoltzmannGibbs}.
Here the normalization is useful because the associated aggregation function
passes explicitly through the concave, affine, and convex cases.

\begin{example}[Entropy-normalized power family]
	\label{ex:entropy-normalized-power}
	For \(r>0\), \(r\neq1\), define
	\[
	\psi_r(t)=\frac{t^r-t}{1-r},
	\qquad 0<t<1,
	\]
	and set \(\psi_1(t)=-t\log t\).  Then
	\[
	\lim_{r\to1}\psi_r(t)=\psi_1(t),
	\qquad
	\psi_r''(t)=-rt^{r-2}<0,
	\]
	and
	\[
	\eta_{\psi_r}(t):=\psi_r(t)-t\psi_r'(t)=t^r,
	\]
	with the convention \(\eta_{\psi_1}(t)=t\).  Consequently,
	\(\eta_{\psi_r}\) is concave for \(0<r<1\), affine at \(r=1\), and convex
	for \(r>1\).
\end{example}

\begin{proof}
	The limit follows from l'Hopital's rule.  For \(r\neq1\),
	\[
	\psi_r'(t)=\frac{rt^{r-1}-1}{1-r},
	\qquad
	\psi_r''(t)=-rt^{r-2},
	\qquad
	\psi_r(t)-t\psi_r'(t)=t^r.
	\]
	For \(r=1\), \(\psi_1'(t)=-1-\log t\), hence
	\(\eta_{\psi_1}(t)=t\).  The sign of
	\(\eta_{\psi_r}''(t)=r(r-1)t^{r-2}\) gives the final assertion.
\end{proof}

\begin{remark}[Relation with the unnormalized power families]
	On the simplex, \(\sum_i\psi_r(\mu_i)\) differs by an additive constant and a
	positive factor from \(\sum_i\mu_i^r\) when \(0<r<1\), and from
	\(-\sum_i\mu_i^r\) when \(r>1\).  Thus this normalization preserves the
	power-family mechanisms while making the entropy limit explicit.
\end{remark}

\section{Conclusion and open problems}
\label{sec:conclusion-open-problems}

We have studied a family of long-only preimage regions on the probability
simplex arising from symmetric separable generated maps.  The main point is that the
concavity of the underlying generating function does not by itself determine
the convexity or barycentric star-shapedness of the long-only preimage region.  In the
symmetric separable case, the defining inequalities split into a coordinate
term and an aggregation term, and this decomposition controls the geometry.

Positive second-derivative behavior of the aggregation function can produce
nonconvexity, while concavity of the aggregation function, or more generally a
second-derivative domination condition, gives barycentric star-shapedness.
Entropy is the affine limiting case in which the aggregation term is constant
and the constraints reduce to coordinate thresholds.

Two directions remain natural.  First, one can develop a boundary theory for
these long-only preimage regions, classifying tangent directions that enter, remain
tangent to, or exit the feasible set.  Second, one can seek analogous
mechanisms beyond the symmetric separable class, where the
coordinate--aggregation separation is no longer available in this form.

\appendix
\section{Proof of the endpoint constraints dominate the midpoint constraint lemma}
\label{app:proof-of-endpoint-constraints-dominate-the-midpoint-constraint}
\begin{proof}
	By Lemma \ref{lem:jensen-difference-general} and Taylor expansion at \(b\),
	\begin{equation}\label{eq:S2-midpoint-jensen-expansion}
		S_2(p_\varepsilon)-S_2(m_\varepsilon)
		=
		J_{\eta_\psi}\left(d_\varepsilon,\frac{3\varepsilon}{2}\right)
		=
		\frac94\eta_\psi''(b)\varepsilon^2+o(\varepsilon^2)>0.
	\end{equation}
	Hence \(S_1(p_\varepsilon)=S_2(p_\varepsilon)>S_2(m_\varepsilon)\).  For the
	remaining constraints at \(p_\varepsilon\), Taylor expansion at \(b\) gives
	\[
	S_3(p_\varepsilon)-S_2(m_\varepsilon)
	=-3\psi''(b)\varepsilon+O(\varepsilon^2)>0,
	\]
	and, for \(\ell\ge4\),
	\[
	S_\ell(p_\varepsilon)-S_2(m_\varepsilon)
	=-\psi''(b)\varepsilon+O(\varepsilon^2)>0.
	\]
	The same estimates apply to \(q_\varepsilon\) by permutation symmetry.
\end{proof}

\section{Proof of the nonempty threshold interval lemma}
\label{app:proof-of-nonempty-threshold-interval-lemma}

\begin{proof}
	Let \(S_0=n\psi(b)\).  We first identify the smallest endpoint
	constraints.  At \(p_\varepsilon\), the first two coordinates are equal to
	\(a_\varepsilon\), so
	\[
	S_1(p_\varepsilon)=S_2(p_\varepsilon).
	\]
	For the third coordinate,
	\[
	S_2(p_\varepsilon)-S_3(p_\varepsilon)
	=
	\psi'(a_\varepsilon)-\psi'(c_\varepsilon)
	=
	3\psi''(b)\varepsilon+O(\varepsilon^2)<0,
	\]
	and for \(\ell\ge4\),
	\[
	S_2(p_\varepsilon)-S_\ell(p_\varepsilon)
	=
	\psi'(a_\varepsilon)-\psi'(b)
	=
	\psi''(b)\varepsilon+O(\varepsilon^2)<0.
	\]
	Thus
	\[
	\min_iS_i(p_\varepsilon)=S_1(p_\varepsilon)=S_2(p_\varepsilon).
	\]
	By symmetry,
	\[
	\min\left\{\min_iS_i(p_\varepsilon),\min_iS_i(q_\varepsilon)\right\}
	=
	S_1(p_\varepsilon)=S_2(p_\varepsilon)
	=
	S_2(q_\varepsilon)=S_3(q_\varepsilon).
	\]
	
	Next we compare these endpoint constraints with the positive-region values.
	Taylor expansion at \(b\) gives
	\[
	S(p_\varepsilon)=S(q_\varepsilon)
	=
	S_0+3\psi''(b)\varepsilon^2+O(\varepsilon^3),
	\]
	\[
	S(m_\varepsilon)
	=
	S_0+\frac34\psi''(b)\varepsilon^2+O(\varepsilon^3),
	\]
	while
	\[
	S_2(p_\varepsilon)
	=
	S_0+\psi''(b)\varepsilon+O(\varepsilon^2).
	\]
	Since \(\psi''(b)<0\), the value \(S_2(p_\varepsilon)\) lies below
	\(S_0\) by order \(\varepsilon\), whereas the positive-region values differ
	from \(S_0\) only by order \(\varepsilon^2\).  Hence, for sufficiently small
	\(\varepsilon>0\),
	\[
	S(p_\varepsilon),\ S(q_\varepsilon),\ S(m_\varepsilon)
	>
	S_2(p_\varepsilon).
	\]
	Therefore \(C_\varepsilon=S_2(p_\varepsilon)\).
	
	Finally,
	\[
	C_\varepsilon-S_2(m_\varepsilon)
	=
	S_2(p_\varepsilon)-S_2(m_\varepsilon).
	\]
	By \eqref{eq:S2-midpoint-jensen-expansion}, this equals
	\(J_{\eta_\psi}(d_\varepsilon,3\varepsilon/2)\), is positive for all
	sufficiently small \(\varepsilon>0\), and has asymptotic size
	\(\frac94\eta_\psi''(b)\varepsilon^2+o(\varepsilon^2)\).  Thus
	\(\mathcal I_\varepsilon\) is nonempty and has the claimed length.
\end{proof}
\section{Proof of the barycentric difference identities lemma}
\label{app:proof-barycentric-difference-identities}
\begin{proof}
	For a \(C^2\) function \(f\), define
	\(F_f(s):=\sum_{i=1}^nf(b+s\delta_i)\) for \(0\le s\le1\).  Since
	\(\sum_i\delta_i=0\), one has \(F_f'(0)=0\), and
	\[
	F_f''(s)=\sum_i\delta_i^2 f''(b+s\delta_i).
	\]
	Hence
	\[
	F_f(\lambda)-F_f(1)
	=-\int_\lambda^1\int_0^sF_f''(u)\,du\,ds.
	\]
	Taking \(f=\psi\) gives \eqref{eq:exact-S-barycentric-identity}, and taking
	\(f=\eta_\psi\) gives the aggregation part of
	\eqref{eq:exact-Theta-barycentric-identity}.  Finally,
	\[
	\psi'((\mu^\lambda)_{(1)})-\psi'(\mu_{(1)})
	=
	\int_\lambda^1\delta_*[-\psi''(b+s\delta_*)]ds,
	\]
	which yields \eqref{eq:exact-Theta-barycentric-identity}.  Concavity makes
	all integrands nonnegative.
\end{proof}
\section{Proof of the convex long-only region with respect to entropy function}
\label{app:proof-convex-entropy}

\begin{proof}
	The set \(D_{\Phi_{-T}^H}\) is a strict superlevel set of the concave function
	\(H\), hence is open and convex.  Since
	\[
	\partial_iH(\mu)=-1-\log\mu_i,
	\qquad
	\sum_j\mu_j\partial_jH(\mu)=-1+H(\mu),
	\]
	we obtain
	\[
	G_i(\mu)
	=
	\Phi_{-T}^H(\mu)+\partial_i\Phi_{-T}^H(\mu)-\sum_j\mu_j\partial_j\Phi_{-T}^H(\mu)
	=
	-\log\mu_i-T.
	\]
	Thus \(G_i(\mu)\ge0\) is equivalent to \(\mu_i\le e^{-T}\), giving the stated
	formula for \(\mathcal L_T\).  Convexity follows from the convexity of
	\(D_{\Phi_{-T}^H}\) and the coordinate half-space constraints.  Nonemptiness
	follows from \(H(\bar e)=\log n>T\) and \(1/n<e^{-T}\).
	
	To see that the full positive region is not long-only, let
	\[
	\nu(t)=\left(t,\frac{1-t}{n-1},\dots,\frac{1-t}{n-1}\right).
	\]
	At \(t_0=e^{-T}\),
	\[
	H(\nu(t_0))+\log t_0
	=
	-(1-t_0)\log\frac{1-t_0}{(n-1)t_0}>0,
	\]
	so \(H(\nu(t_0))>T\).  Hence, for \(t>t_0\) sufficiently close to \(t_0\), we
	still have \(\nu(t)\in D_{\Phi_{-T}^H}\), but
	\(G_1(\nu(t))=-\log t-T<0\).  Thus short selling occurs somewhere in
	\(D_{\Phi_{-T}^H}\).
\end{proof}
\newpage
\bibliographystyle{plain}
\bibliography{refs}

\end{document}